\documentclass[final,12pt]{amsart}
\usepackage{amssymb, amsmath, amsthm}
\usepackage{mathrsfs}  
\usepackage{graphicx}
\usepackage[colorlinks=true, citecolor=blue, linkcolor=blue, urlcolor=blue]{hyperref}
\usepackage[text={5.5in,8in},centering]{geometry}
 \usepackage{xcolor}

\usepackage{mathrsfs}  

\usepackage{bm}

\usepackage{ifdraft}
\ifoptionfinal{
\usepackage[disable]{todonotes}
}{
\usepackage[norefs, nocites]{refcheck}

\usepackage[notref, notcite]{showkeys}
\usepackage[bordercolor=white, color=white]{todonotes}
}

\makeatletter\providecommand\@dotsep{5}\def\listtodoname{List of Todos}\def\listoftodos{\hypersetup{linkcolor=black}\@starttoc{tdo}\listtodoname\hypersetup{linkcolor=blue}}\makeatother

\usepackage{hyperref}
\usepackage{etoolbox}

\usepackage[colored]{shadethm} 
\definecolor{shaderulecolor}{rgb}{0.651,0.074,0.090}

\newshadetheorem{theorem}{Theorem}[section]
\newtheorem{lemma}[theorem]{Lemma}
\newtheorem{proposition}[theorem]{Proposition}

\theoremstyle{definition}
\newtheorem{definition}[theorem]{Definition}
\theoremstyle{remark}
\newtheorem{remark}[theorem]{Remark}

\def\C{\mathbb C}
\def\R{\mathbb R}

\def\N{\mathbb N}

\def\O{\mathcal O}

\renewcommand{\leq}{\leqslant}
\renewcommand{\geq}{\geqslant}

\def\p{\partial}
\DeclareMathOperator{\supp}{supp}

\newcommand*\xbar[1]{%
   \hbox{%
     \vbox{%
       \hrule height 0.5pt 
       \kern0.5ex
       \hbox{%
         \ensuremath{#1}%
       }%
     }%
   }%
} 

\title[Fractional Calder\'{o}n problem on Riemannian manifolds]{Fractional Calder\'{o}n problem on a closed Riemannian manifold}
\author[Ali Feizmohammadi]{Ali Feizmohammadi}
\address{Fields institute, 222 College St, Toronto, ON M5T 3J1}
\email{afeizmoh@fields.utoronto.ca}

\begin{document}


\maketitle
\begin{abstract}
Given a fixed $\alpha \in (0,1)$, we study the inverse problem of recovering the isometry class of a smooth closed and connected Riemannian manifold $(M,g)$, given the knowledge of a source-to-solution map for the fractional Laplace equation $(-\Delta_g)^\alpha u=f$ on the manifold subject to an arbitrarily small observation region $\mathcal O$ where sources can be placed and solutions can be measured. This can be viewed as a non-local analogue of the well known anisotropic Calder\'{o}n problem that is concerned with the limiting case $\alpha=1$. While the latter problem is widely open in dimensions three and higher, we solve the non-local problem in broad geometric generality, assuming only a local property on the a priori known observation region $\mathcal O$ while making no geometric assumptions on the inaccessible region of the manifold, namely $M\setminus \mathcal O$. Our proof is based on discovering a hidden connection to a variant of Carlson's theorem in complex analysis that allows us to reduce the non-local inverse problem to the Gel'fand inverse spectral problem. 
\end{abstract}

\section{Introduction and the main result}
\subsection{Fractional Laplace equation} Let $(M,g)$ be a smooth closed and connected Riemannian manifold with dimension $n\geq 2$. Here, by closed we mean that the manifold is compact and without boundary. We denote by $-\Delta_g$ the (positive) Laplace--Beltrami operator on $(M,g)$ that is defined in local coordinates via the expression
\begin{equation}
\label{laplace}
-\Delta_gu=-\frac{1}{\sqrt{\det g}}\sum_{j,k=1}^n\frac{\p}{\p x^j}\left(\sqrt{\det g}\,g^{jk}\frac{\p u}{\p x^k}\right)\quad \forall\, u\in \mathcal C^{\infty}(M).
\end{equation}  
We write $(\cdot,\cdot)_{L^2(M)}$ for the standard Hermitian inner product on $L^2(M)$, write $0=\lambda_0<\lambda_1<\lambda_2<\ldots$ for the distinct eigenvalues of $-\Delta_g$ written in strictly increasing order. For each $k=0,1,2,\ldots$ we write $d_k$ to stand for the multiplicity of eigenvalue $\lambda_k$ and  let $\phi_{k,1},\ldots,\phi_{k,d_k}$ denote an $L^2(M)$-orthonormal basis for the eigenspace corresponding to $\lambda_k$. Finally, given any $k=0,1,\ldots$ we define $\pi_{k}:L^2(M)\to L^2(M)$ to be the projection operator on the eigenspace of $\lambda_k$ that is defined via
\begin{equation}
\label{proj_op}
\pi_k f = \sum_{\ell=1}^{d_k}(f,\phi_{k,\ell})_{L^2(M)}\,\phi_{k,\ell}.
\end{equation}
Given a fixed $\alpha\in (0,1)$ and any $u \in \mathcal C^{\infty}(M)$ the fractional Laplacian of $u$ with order $\alpha$, denoted by $(-\Delta_g)^{\alpha}u$, is typically defined on a closed manifold $(M,g)$ via the spectral expression 
\begin{equation}
\label{alpha_laplace}
(-\Delta_g)^{\alpha}u=\sum_{k=1}^{\infty}\lambda_k^{\alpha}\,\pi_ku,\quad \text{on $M$}.
\end{equation}
The fractional Laplace--Beltrami operator $(-\Delta_g)^{\alpha}$ can be extended by continuity as a bounded linear map from $H^s(M)$ to $H^{s-2\alpha}(M)$ for any $s\geq 0$.  We refer the reader to \cite[Theorem 1.2]{ACM} for an integral representation of the definition \eqref{alpha_laplace} stated on compact Riemannian manifolds (with or without boundary) that makes connections to alternative definitions of the fractional Laplace--Beltrami operator in the pseudo-differential framework, as well as providing applications in proving fractional Sobolev embedding theorems. 

In this paper, we consider the following elliptic equation,
\begin{equation}
\label{pf}
(-\Delta_g)^{\alpha}u=f \quad \text{on $M$}.
\end{equation}
Give any $s\geq 0$ and any $f \in H^s(M)$ with $(f,1)_{L^2(M)}=0$, equation \eqref{pf} admits a unique solution $u\in H^{s+2\alpha}(M)$ that is given by the expression
\begin{equation}
\label{spectral_sol} 
u=\sum_{k=1}^{\infty} \lambda_k^{-\alpha}\,\pi_kf\quad \text{on $M$}.
\end{equation}
We remark that the necessary orthogonality condition $(f,1)_{L^2(M)}=0$ above arises as $\lambda_0=0$ is an eigenvalue on a closed Riemannian manifold with the corresponding eigenspace of constant functions. 

\subsection{The inverse problem and our main result} To formulate our inverse problem associated to \eqref{pf} with a fixed $\alpha$, it will be convenient to first consider a smooth submanifold $(\O,g|_{\O})$ with codimension zero that has a smooth boundary, representing an {\em observable} region in the manifold that is assumed to be a priori known. This region can be arbitrarily small and it merely allows us to define a {\em source-to-solution} map corresponding to \eqref{pf} subject to placing sources $f$ that are supported in $\O$ and subsequently measuring the solutions in the same observable region $\O$. We will write $\O^{\textrm{int}}$ for the interior of the set $\O$ and $\p \O$ for its boundary, so that $\O=\O^{\textrm{int}}\cup \p\O$. Also, given an open set $U$, $\mathcal C^{\infty}_0(U)$ denotes the space of complex valued smooth functions with a compact support in $U$.

\begin{definition}[Source-to-Solution map]
	\label{so_to_sol_def}
Let $\alpha\in (0,1)$ and let $\O$ be a smooth submanifold of $(M,g)$ with codimension zero and a smooth boundary. Given any 
$$f \in \mathcal C^{\infty}_0(\O^{\textrm{{int}}})\quad \text{with $(f,1)_{L^2(\O)}=0$},$$
we define
\begin{equation}\label{so_to_sol}
\mathscr L^{\alpha,\O}_{M,g}f= u{|_{\O}},\end{equation}
where $u|_{\O}$ is the restriction to the set $\O$ of the unique solution to \eqref{pf} with source $f$. 
\end{definition}

Our inverse problem can now be stated as follows: Does the knowledge of the source-to-solution map, $\mathscr L^{\alpha,\O}_{M,g}$, corresponding to a fixed fractional order $\alpha$ and a fixed observation region $\O$ determine the ambient manifold $(M,g)$ globally? We note that there is a natural obstruction to uniqueness that comes from isometries on $M$ that fix the observable set $\O$. In the case of the anisotropic Calder\'{o}n problem, that concerns the limiting case $\alpha=1$, this obstruction was first noted by Luc Tartar \cite{KV}. In the non-local case we refer the reader to \cite[Theorem 4.2]{GhUh}. Therefore, for our inverse problem that we call the fractional Calder\'{o}n inverse problem, the best we can expect is to recover the manifold up to this obstruction.  Let us remark that for the anisotropic Calder\'{o}n inverse problem and specifically when the dimension of the manifold is two, there is an additional obstruction to uniqueness coming from conformal invariance of the Laplace--Beltrami operator in two dimensions \cite{LaUh,Na96}. However, this additional obstruction is not present in the fractional Calder\'{o}n inverse problem since equation \eqref{alpha_laplace} is not conformally invariant in any dimension $n\geq 2$.

Our main result states that it is possible to uniquely recover the isometry class of the manifold $(M,g)$ from the source-to-solution map, assuming only a local property on the observable set $\O$. In order to state our assumption, we need to recall the definition of Gevrey classes $\mathscr G^{N}$ with $N\geq 1$ which are intermediate functional spaces lying between the space of real-analytic functions and the space of smooth functions, see for example \cite{Gr,Rudino}. We also refer the reader to \cite{BT,Carleman} for connections between Gevrey classification and the Denjoy--Carleman classification of smooth functions.

\begin{definition}
	\label{C_N_def}
	Given an open set $U\subset \R^n$ and a constant $N\geq 1$, we say that a smooth function $f$ belongs to the Gevrey class $\mathscr G^{N}(U)$ if given any compact set $K\Subset U$, there is a constant $A>0$ depending on $K$ and $f$ such that
	$$ \|\partial^{\beta}f\|_{L^{\infty}(K)}\leq A^{|\beta|+1}\,(\beta!)^{N},$$
	for all multi-indexes $\beta\in \{0,1,2,\ldots\}^n$. 
\end{definition}
Note that $\mathscr G^{N_1}\subset \mathscr G^{N_2}$ whenever $N_1<N_2$. It is also clear from the latter definition that real-analytic functions are a member of all Gevrey classes. We are now ready to state the local assumption that we will impose on our observation set $\O$.

\begin{definition}
	\label{Gevrey_metric}
	Let $N\geq 1$ and let $\O$ be a smooth submanifold of $(M,g)$ with codimension zero and a smooth boundary. We say that $\O$ locally admits a metric in $\mathscr G^{N}$, if given any $p\in \O^{\textrm{{int}}}$ there is a coordinate chart $(U_p,\psi_p)$ belonging to its maximal smooth atlas with $p\in U_p\subset \O^{\textrm{{int}}}$ such that the components of the metric $g|_{U_p}$, when expressed in this coordinate chart, belong to the Gevrey class $\mathscr G^N(\psi_p(U_p))$. 
\end{definition}

Our main result can now be stated as follows.
\begin{theorem}
\label{main_thm}
Let $\alpha \in (0,1)$. For $j=1,2$, let $(M_j,g_j)$ be a smooth closed and connected Riemannian manifold with dimension $n\geq 2$ and let $\O_j\subset M_j$ be a smooth submanifold of codimension zero with a smooth boundary. Assume that $(\O_j,g_{j}|_{\O_j})$, $j=1,2$, locally admits a metric in $\mathscr G^{N}$ for some $N\in [1,2),$
in the sense of Definition~\ref{Gevrey_metric} and that  
\begin{equation}
\label{O_eq}
(\O_1,g_1|_{\O_1})=(\O_2,g_2|_{\O_2}).
\end{equation}
Suppose that
\begin{equation}
\label{so_to_sol_eq}
\mathscr L^{\alpha,\O_1}_{M_1,g_1} f=\mathscr L^{\alpha,\O_2}_{M_2,g_2} f\quad \text{for all $f \in \mathcal C^{\infty}_0(\O_1^{\text{{\em int}}})$ with $(f,1)_{L^2(\O_1)}=0$.}
\end{equation}
Then there exists a diffeomorphism $\Phi:M_1\to M_2$ that fixes the set $\O_1$, such that
$$\Phi^*g_2=g_1\quad \text{on $M_1$}.$$
\end{theorem}

\begin{remark}
	Note that the assumption that the observation region is locally in some Gevrey class $\mathscr G^N$ is naturally a weaker assumption than real analyticity, so as a corollary let us mention that if the observation region is assumed to be real-analytic, the above theorem gives a uniqueness result for recovery of the isometry class of $(M,g)$. Let us also remark that the equality of the observation regions \eqref{O_eq} can be relaxed; it suffices that they are isometric to each other.
\end{remark}

The main novelty of Theorem~\ref{main_thm} is that we recover the ambient Riemannian geometry for the equation \eqref{pf}. While we assume that the observation region $\O$ locally admits a metric in some Gevrey class, we make no geometric assumptions on the inaccessible region of the manifold, namely $M\setminus \O$. We achieve this by discovering a connection between the fractional Calder\'{o}n problem and a variant of Carlson's classic analytic continuation theorem in the complex plane that is due to Pila \cite{Pila}, see Appendix~\ref{sec_appendix} for the details. This allows us to fully recover the spectral data associated to the Dirichlet Laplacian of the inaccessible manifold and reduce our problem to the well known Gel'fand inverse spectral problem. The result then follows from the work of Kurylev and Belishev in \cite{BK92} that uses the Boundary Control method to solve the latter problem in full geometric generality. 

\subsection{Motivation}

Let us first motivate our inverse problem by remarking that the study of fractional orders of elliptic operators is a topic with many physical applications. Indeed, fractional orders of elliptic operators appear in the modeling of non-local diffusion, permeable medium, and fluid turbulence in physics \cite{GL97,Las00,ZD10}, as well as modeling of L\'{e}vy processes, population dynamics and jump processes in stock markets \cite{AB88,Lev}. We refer the reader to the survey article \cite{BV16} for more applications. 

Aside from physical motivations, let us also mathematically motivate our inverse problem associated to the non-local equation \eqref{pf} by making a few remarks about the limiting case $\alpha=1$ (which is outside the scope of this paper). When $\alpha=1$, our inverse problem can be seen as an equivalent formulation of the well known anisotropic Calder\'{o}n problem, usually stated for manifolds with boundary and in terms of the Dirichlet-to-Neumann map as opposed to source-to-solution maps, see for example \cite{DKSU}. The anisotropic Calder\'{o}n problem has been solved in dimension two \cite{ALP05,LU,Na96} but remains a major open problem in higher dimensions. The works \cite{LLS,LTU,LaUh,LU} study the anisotropic Calder\'{o}n problem on real-analytic manifolds with a real-analytic metric and an analogous result is proved in \cite{GS} for Einstein manifolds (which are real-analytic in their interior). Outside the real-analytic category results are rather sparse. The seminal work \cite{SU} provides a uniqueness result for conformally Euclidean metrics on a bounded domain using the idea of Complex Geometric Optics solutions. The works \cite{DKSU,DKLS} generalize this idea further and show uniqueness for the metric holds assuming that (i) the conformal class of the manifold is a priori known and (ii) the manifold has, within its conformal class, a global unit parallel vector field. We refer the reader to \cite{Gu} and the references therein for a survey of this problem.
\subsection{Previous literature}
The study of inverse problems associated to fractional orders $\alpha<1$ of the Laplace operator was initiated in \cite{GSU}. There, the authors solved the inverse problem of recovering a zeroth order coefficient $q$ in the equation
\begin{equation}
\label{schrodinger}
(-\Delta)^\alpha u +q\,u=0 \quad \text{on $\Omega$},\end{equation}
with $u|_{\R^n\setminus\Omega}=f$ given the knowledge of an associated source-to-solution map. We recall that the fractional Laplace operator of order $\alpha$ on the Euclidean space $\R^n$ is defined by $$(-\Delta)^\alpha u=\mathscr F^{-1}(|\xi|^{2\alpha}\hat{u}(\xi)),$$
where $\mathscr F u =\hat{u}$ stands for the Fourier transform of $u$. Their proof relied on a strong unique continuation result for the fractional Laplace operator together with a density argument for products of solutions to equation \eqref{schrodinger} that is based on a Runge type approximation argument, see for example \cite{DSV}. Since then, the area of recovering lower order coefficients (or lower order perturbations) in the context of fractional elliptic equations analogous to \eqref{schrodinger} has been a very active area of research. We mention for example the following results: uniqueness and reconstruction methods subject to a single measurement \cite{GRSU}, recovery of zeroth order coefficients in the presence of a priori known anisotropic leading order coefficients \cite{GLX}, study of uniqueness and (in-)stablity in the low regularity regime \cite{RS18,RS20}, study of inversion methods based on monotonicity assumptions \cite{HL19,HL20}, study of  uniqueness for magnetic problems and their lower perturbations \cite{BGU,CLR,Cov20a} and study of uniqueness in the presence of nonlinearity \cite{LL20,LO20}.  We refer the reader to the works \cite{Ru18,Salo} for review. 

Let us mention that all of the preceding works are concerned with the recovery of lower order coefficients (or lower order perturbations) in non-local equations and assume that the leading order coefficients are known. In the recent work \cite{GhUh}, the authors consider fractional powers of general self-adjoint positive operators on a domain $\Omega\subset \R^n$ with unknown anisotropic leading order coefficients in $\Omega$. They prove that the exterior data for the associated non-local operators determines the Cauchy data set on the boundary $\p \Omega$ for the corresponding local operator and as a corollary, they prove several new results for non-local inverse problems by using the corresponding resuts for the local inverse problems. Examples include recovery of a general Riemannian metric up to the natural gauge in dimension two and recovery of Riemannian metrics on $\Omega\subset \R^n$, $n\geq 2$ that are conformally Euclidean as well as recovery of real-analytic metrics up to the natural gauge.

\subsection*{Acknowledgments} The author was supported by the Fields institute for research in mathematical sciences. The author gratefully acknowledges the talk of Gunther Uhlmann that presented the results of \cite{GhUh} in the inverse problems and nonlinearity conference in August 2021 (and where an analogous inverse problem was posed on a Riemannian manifold with boundary), as a main source of motivation for this work. 

\section{Analytic interpolation in the complex plane}
\label{sec_carlson}
Our main aim in this section is to prove the following proposition.
\begin{proposition}
\label{main_prop}
Let the hypotheses of Theorem~\ref{main_thm} be satisfied. For $j=1,2$ let $\{\lambda_k^{(j)}\}_{k=0}^{\infty}$ be the set of distinct eigenvalues of $-\Delta_{g_j}$ on $(M_j,g_j)$ written in strictly increasing order and let $\pi_k^{(j)}$ be the projection operator on the eigenspace of $\lambda^{(j)}_k$ associated to the operator $-\Delta_{g_j}$ on $(M_j,g_j)$. Then, there holds
$$ \lambda_k^{(1)}=\lambda_k^{(2)} \quad \text{and}\quad (\pi_k^{(1)}f)|_{\O_1}=(\pi^{(2)}_kf)|_{\O_1},\quad \text{for all $k=0,1,2,\ldots$}$$
and all $f \in \mathcal C^{\infty}_0(\O_1^{\textrm{{\em int}}})$. Here, $|_{\O_1}$ is the restriction to the set $\O_1=\O_2$.
\end{proposition}

Before proving Proposition~\ref{main_prop} we need to fix a few notations. In view of the equality \eqref{O_eq} we use the notation 
\begin{equation}
\label{O_g_equal}
\O:=\O_1=\O_2\quad \text{and}\quad g:=g_1|_{\O}=g_2|_{\O}.
\end{equation}
For $j=1,2$, and given any $f\in \mathcal C^{\infty}_0(\O^{\textrm{int}})$ we (for now, formally) define the function 
$$\zeta_f^{(j)}: \C\times \O \to \C,$$
by the expression
\begin{equation}
\label{zeta_def}
\zeta^{(j)}_f(z,x)= \sum_{k=1}^{\infty} (\lambda_k^{(j)})^{z}\,(\pi_k^{(j)}f)(x),\quad \forall\, (z,x) \in \C \times \O,
\end{equation}
where $(\lambda_k^{(j)})^{z}=e^{z\log(\lambda_k^{(j)})}$ for any $k\in \N$ and we recall that
\begin{equation}
\label{proj_j_def}
(\pi_k^{(j)} f)(x) = \sum_{\ell=1}^{d^{(j)}_k}(f,\phi^{(j)}_{k,\ell})_{L^2(M_j)}\,\phi^{(j)}_{k,\ell}(x),\quad \forall\, x\in \O,
\end{equation}
with $\{\phi^{(j)}_{k,\ell}\}_{\ell=1}^{d_k^{(j)}}$ denoting an $L^2(M_j)$-orthonormal basis for the eigenspace corresponding to $\lambda_k^{(j)}$. The following lemma justifies the definition \eqref{zeta_def} showing not only that the infinite sum in \eqref{zeta_def} is absolutely convergent but that it is also an entire holomorphic function.

\begin{lemma}
	\label{lem_zeta_hol}
	For $j=1,2$, given any fixed $f\in \mathcal C^{\infty}_0(\O^{\textrm{{\em int}}})$ and any $x \in \O$, the function $$\zeta_f^{(j)}(\cdot,x):\C\to \C,$$ 
	is an entire holomorphic function.
\end{lemma}

\begin{proof}
	Given $j=1,2$ and any $\mu\in \N$, let us define the partial sums
	\begin{equation}
	\label{S_mu}
	S_{\mu}^{(j)}(z,x)=\sum_{k=1}^\mu(\lambda_k^{(j)})^{z}(\pi_k^{(j)}f)(x),\quad z \in \C,\quad x\in \O,\end{equation}
	and
	\begin{equation}
	\label{S_mu_abs}
	\widetilde{S}_{\mu}^{(j)}(z,x)=\sum_{k=1}^\mu|(\lambda_k^{(j)})^{z}|\,|(\pi_k^{(j)}f)(x)|,\quad z \in \C,\quad x\in \O,\end{equation}
	where $|\cdot|$ denotes the absolute value function on complex numbers. Recall that for $j=1,2$ and given any $\ell=1,\ldots,d^{(j)}_k$, we have 
	$$-\Delta_{g_j}\phi^{(j)}_{k,\ell}=\lambda_k^{(j)}\phi^{(j)}_{k,\ell}\quad \text{on $M_j$,}\qquad j=1,2.$$
	Using the latter identity together with the fact that $f$ is compactly supported in $\O^{\textrm{int}}$, we can repeatedly apply Green's integral identity to deduce that
	$$(f,\phi^{(j)}_{k,\ell})_{L^2(\O)}=\int_{\O}f\,\phi^{(j)}_{k,\ell}\,dV_{g}=(-1)^{m}(\lambda_k^{(j)})^{-m}\int_{\O}(\Delta^{(m)}_{g}f)\,\phi^{(j)}_{k,\ell}\,dV_{g},$$
	for any $m\in \N$, where 
	\begin{equation}
	\label{Delta_m}
	\Delta_{g}^{(m)}f= \underbrace{\Delta_{g}\circ \ldots\circ \Delta_{g}}_{\text{$m$ times}}f,\end{equation}
	and we recall that $g$ is as defined by \eqref{O_g_equal}.  Applying the Cauchy--Schwarz inequality and recalling the definition of the projection operator it follows that given $j=1,2$, any $k\in \N$ and $m\in \N$, there holds
	\begin{equation}
	\label{proj_est_1} \|\pi_k^{(j)}f\|_{L^{\infty}(\O)}\leq (\lambda_k^{(j)})^{-m}\|\Delta_{g}^{(m)}f\|_{L^{2}(\O)}\left(\sum_{\ell=1}^{d^{(j)}_k}\|\phi_{k,\ell}^{(j)}\|_{L^{\infty}(\O)} \right).
\end{equation}
	Next let us define
	$$  
	\kappa(z)= \begin{cases}
	\lceil \textrm{Re}(z)\rceil +n+1 \quad & \text{if $\textrm{Re}(z)\geq 0$},\\
	n+1 \quad & \text{if $\textrm{Re}(z)<0$,}
	\end{cases}
	$$
	where $\textrm{Re}(z)$ denotes the real part of $z\in \C$ and $\lceil\cdot\rceil$ is the ceiling function that gives the smallest integer not less than its argument. Applying the inequality \eqref{proj_est_1} with the choice $m=\kappa(z)$, it follows that given any $x\in \O$
	\begin{equation}
	\label{S_mu_bound}
	\widetilde{S}^{(j)}_\mu(z,x) \leq C_{1,j}\|\Delta_{g}^{(\kappa(z))}f\|_{L^\infty(\O)}\sum_{k=1}^{\mu}\frac{\sum_{\ell=1}^{d^{(j)}_k}\|\phi_{k,\ell}^{(j)}\|_{L^{\infty}(\O)}}{\lambda_k^{n+1}},\end{equation}
	for some constant $C_{1,j}>0$ that only depends on $j=1,2$. Next, we recall the following uniform bound for eigenfunctions in a closed Riemannian manifold (see \cite{Hor}),
	\begin{equation}
	\label{uni_eigen} \|\phi^{(j)}_{k,\ell}\|_{L^{\infty}(M_j)}\leq C_{2,j}\,(\lambda_k^{(j)})^{\frac{n-1}{4}}\quad \text{for $\ell=1,\ldots,d_k$},
	\end{equation}
	for some constant $C_{2,j}>0$ only depending on $j=1,2$. Recall also from the classical Weyl's law on smooth closed and connected Riemannian manifolds (see \cite{Ivrii} for example) that given any $\lambda>0$, there holds
	\begin{equation}
	\label{weyl} 
	\mathcal N_j(\lambda)\leq C_{3,j}\,\lambda^{\frac{n}{2}},
	\end{equation}
	for some constant $C_{3,j}>0$ only depending on $j=1,2$, where $\mathcal N_j(\lambda)$ is the number of eigenvalues of $-\Delta_{g_j}$ on $(M_j,g_j)$ that are less than $\lambda$ (counting with multiplicity). Combining the estimates \eqref{uni_eigen}--\eqref{weyl} we conclude that there exists $C_{4,j}>0$ only depending on $j=1,2$, such that
	\begin{equation}
	\label{eigen_inf_sum}
	\sum_{k=1}^{\mu}\frac{\sum_{\ell=1}^{d^{(j)}_k}\|\phi_{k,\ell}^{(j)}\|_{L^{\infty}(\O)}}{\lambda_k^{n+1}}\leq C_{4,j} \sum_{k=1}^{\infty} \frac{1}{k^{\frac{n}{4}+1}}=:C_{5,j}<\infty.
	\end{equation}
	Thus, given any $x\in \O$ and $z\in \C$, there holds: 
	\begin{equation}
	\label{S_bound}
	\widetilde{S}_\mu^{(j)}(z,x)\leq C_j\,\|\Delta_{g}^{(\kappa(z))}f\|_{L^\infty(\O)},
	\end{equation}
	for some constant $C_j>0$ only depending on $j=1,2$. Recalling the definitions \eqref{S_mu}--\eqref{S_mu_abs}, we conclude from the latter estimate that given any $x \in \O$, the function $S_\mu^{(j)}(z,x)$ converges uniformly on compact subsets of $\C$ and 		hence
	\begin{equation}
	\label{S_limit} \zeta_f^{(j)}(z,x)=\lim_{\mu\to\infty}S_\mu^{(j)}(z,x),\end{equation}
	is an entire holomorphic function on $\C$.
	\end{proof}

\begin{remark}
	The above proof in fact shows that given any $f\in \mathcal C^{\infty}_0(\O^{\textrm{int}})$, the mapping 
	$$ z \mapsto \sum_{k=1}^{\infty} (\lambda_k^{(j)})^{z}\,(\pi_k^{(j)}f)|_{\O},$$
	is an entire holomorphic function in the sense of taking values in $\mathcal C^{\infty}(\O)$.
	\end{remark}

\begin{lemma}
	\label{lem_eq_seq}
	For a fixed $\alpha \in (0,1)$, let the strictly monotone sequence $(b_k)_{k=1}^{\infty}$ be defined via
		\begin{equation}
	\label{b_sequence}
	\left\{
	\begin{array}{ll}
	b_{2m}=m & \forall \, m\in \N 
	\\
	b_{2m-1}=m-\alpha, & \forall \, m\in \N 
	\end{array}
	\right.
	\end{equation}
	Given any $f\in \mathcal C^{\infty}_0(\O^{\textrm{{\em int}}})$ there holds:
	\begin{equation}
	\label{zeta_b_equal} \zeta_f^{(1)}(b_k,x)=\zeta_f^{(2)}(b_k,x)\quad \forall\, k\in \N\quad \forall\,x\in \O.
	\end{equation}
\end{lemma}

\begin{proof}
The claim \eqref{zeta_b_equal} for $k=2m$ with $m\in \N$ is trivial and follows from the local property \eqref{O_g_equal}. Indeed, we observe from \eqref{O_g_equal} that
\begin{equation}
\label{k_even}
 \Delta_{g_1}^{(m)}f =  \Delta_{g_2}^{(m)}f=\Delta_{g}^{(m)}f \quad \text{on $\O$},\end{equation}
where given any $j=1,2$ and $m\in \N$ the notation $ \Delta_{g_j}^{(m)}f$ is defined analogously to \eqref{Delta_m}.  Using the spectral representation of $\Delta_{g_j}f$ in terms of the eigenvalues $\{\lambda^{(j)}_k\}_{k=0}^{\infty}$ and projection operators $\{\pi^{(j)}_kf\}_{k=0}^{\infty}$ and recalling that $\lambda_0=0$, the equality \eqref{k_even}, when restricted to the subset $\O$, can be rewritten as
$$\sum_{k=1}^{\infty}(\lambda_k^{(1)})^m\,(\pi_k^{(1)}f)|_{\O}=\sum_{k=1}^{\infty}(\lambda_k^{(2)})^m\,(\pi_k^{(2)}f)|_{\O}.
$$
Recalling the definitions \eqref{zeta_def} and \eqref{b_sequence} yields the claim \eqref{zeta_b_equal} for $k=2m$. For the case $k=2m-1$ with $m\in \N$, the claim \eqref{zeta_b_equal} follows from equality of the source-to-solution maps. Indeed, let us first note that given any $m\in \N$ and $f\in \mathcal C^{\infty}_0(\O^{\textrm{int}})$ it follows from Stoke's theorem that
$$ (\Delta_{g_j}^{(m)}f,1)_{L^2(\O)}=\int_{\O}\Delta_{g_j}^{(m)}f\,dV_{g_j}=0,\qquad j=1,2.$$
Thus, by the hypothesis \eqref{so_to_sol_eq} of Theorem~\ref{main_thm} together with \eqref{k_even}, we have
$$ \mathscr L^{\alpha,\O_1}_{M_1,g_1}(\Delta_{g_1}^{(m)}f) = \mathscr L^{\alpha,\O_2}_{M_2,g_2}(\Delta_{g_2}^{(m)}f).$$  
Using equations \eqref{pf}--\eqref{spectral_sol} together with \eqref{Delta_m}, the latter identity, when restricted to the subset $\O$, can be rewritten as 
$$
\sum_{k=1}^{\infty}(\lambda_k^{(1)})^{m-\alpha}\,(\pi_k^{(1)}f)|_{\O}=\sum_{k=1}^{\infty}(\lambda_k^{(2)})^{m-\alpha}\,(\pi_k^{(2)}f)|_{\O}.
$$
The claim \eqref{zeta_b_equal} with $k=2m-1$ now follows from recalling the definitions \eqref{zeta_def} and \eqref{b_sequence}.
\end{proof}

In what follows, we fix a point $p\in \O^{\textrm{int}}$ and (in view of the hypothesis of Theorem~\ref{main_thm}) let $(U_p,\psi_p)$ be a coordinate chart in $\O$ with $p\in U_p\subset \O^{\textrm{int}}$ and such that the components of the metric $(\psi_p)_*g$ belong to the Gevrey class $\mathscr G^N(\psi_p(U_p))$ where $N$ is as in the hypothesis of Theorem~\ref{main_thm} (independent of $p$). Our aim will be to show that given any $f\in \mathcal C^{\infty}_0(U_p)$ and any $x\in \O$, the holomorphic function $\zeta_f^{(j)}(z,x)$ with $j=1,2$, can be constructively determined from the source-to-solution map $\mathscr L_{M_j,g_j}^{\alpha,\O}$. This will partly rely on construction of suitable sources $f$ that exploit the fact that the components of $(\psi_p)_*g $ are in the Gevrey class $\mathscr G^N(\psi_p(U_p))$ as well as a variant of Carlson's theorem in complex analysis, see Appendix~\ref{sec_appendix}.

As a first step, we need to construct a special family of sources to use as test functions. To this end, we begin by defining a non-negative function $\chi_0\in \mathcal C^{\infty}_0(\R)$ with the properties that $\chi_0$ is compactly supported in the interval $(-1,1)$, that $\chi_0(t)=1$ on $|t|\leq \frac{1}{2}$ and additionally that there exists a constant $c>0$ so that
\begin{equation}
\label{chi_0}
\|\frac{d^k}{dt^k}\chi_0\|_{L^{\infty}(\R)}\leq c^{k}(k!)^{N'},\quad\forall\,k\in \N,
\end{equation}
for some fixed real number $N'$ satisfying
\begin{equation}
\label{N'}
N<N'<2,
\end{equation}
where $N$ as in the hypothesis of Theorem~\ref{main_thm}. Since $N'>1$ (recall that $N\geq 1$), the existence of such a function $\chi_0$ is classical and follows from the Denjoy--Carleman classification of quasi-analytic functions, see for example \cite[Theorem 1.3.8, Theorem 1.4.2]{HorI}. Next, given any fixed $q\in U_p$, let us write $\psi_p(q)=(q_1,q_2,\ldots,q_n)$ and given any $r>0$, let us define 
$$\mathbb B_{q,r}:=\{(y^1,y^2,\ldots,y^n)\in \R^n\,:\, |y-\psi_p(q)|<r\},$$
where $|y-\psi_p(q)|=\sqrt{(y^1-q_1)^2+\ldots+(y^n-q_n)^2}$. We define $\delta_0(q) \in (0,\infty]$ via   
\begin{equation} 
\label{delta_0}
 \delta_0(q):=\sup \{r>0\,:\, \mathbb B_{q,a}\subset \psi_p(U_p) \quad \text{for all $a\in (0,r)$}\}.
 \end{equation}
Subsequently, we define for any $\delta \in (0,\delta_0(q))$ the function 
\begin{equation}
\label{F_def} F_{q,\delta}(x)=\begin{cases}\delta^{-n}\chi_0(\delta^{-1}|\psi_p(x)-\psi_p(q)|)\qquad &\text{if $x\in U_p$},\\
0\qquad &\text{otherwise}
\end{cases}
\end{equation}
and note that $F_{q,\delta}$ can be viewed as a smooth function on the entire manifold that is compactly supported in the set
$$\psi_p^{-1}(\mathbb B_{q,\delta_0(q)}) \subset U_p.$$
It is straightforward to see that the functions $F_{q,\delta}$ can act as mollifiers, that is to say, given any $f \in \mathcal C^{\infty}_0(U_p)$, there holds
$$f(x)=c_0\lim_{\delta\to 0}\int_{\psi_p(U_p)}f(\psi_p^{-1}(y))\,F_{\psi_p^{-1}(y),\delta}(x)\,dy,\qquad \forall\, x\in \O,$$
where $c_0>0$ is an explicit constant only depending on $\|\chi_0\|_{L^1(\R)}$ and $n$. In fact, noting the definition of the projection operators given by \eqref{proj_j_def} and the fact that they map into finite dimensional spaces, it is clear that a similar formula can be written for them, that is to say, given $j=1,2$, any $f\in \mathcal C^{\infty}_0(U_p)$, and any $k\in \N$, there holds
\begin{equation}
\label{proj_F_f}
(\pi_k^{(j)}f)(x)=c_0\lim_{\delta\to 0}\int_{\psi_p(U_p)}f(\psi_p^{-1}(y))\,(\pi_k^{(j)}F_{\psi_p^{-1}(y),\delta})(x)\,dy,
\end{equation} 
for all $x\in \O$.
 
\begin{lemma}
	\label{technical_lemma}
	Let $q\in U_p$ and let $\delta_0(q)>0$ be as defined by \eqref{delta_0}. Then given any $0<\delta<\delta_0(q)$ and any $m\in \N$, there holds
	$$ \|\Delta_g^{(m)} F_{q,\delta}\|_{L^\infty(U_p)}\leq C^{m}\,((2m)!)^{N'},$$
	for some constant $C>0$ independent of $m$, where $F_{q,\delta}$ is as defined by \eqref{F_def} and $
	\Delta_g^{(m)}
	\cdot$ is defined analogously to \eqref{Delta_m}. 
\end{lemma}

\begin{proof}
	Note that as $F_{q,\delta}$ is compactly supported in $U_p$, it suffices to prove the estimate in the local coordinate system $(U_p,\psi_p)$. Throughout the remainder of this proof and for the sake of brevity, we will slightly abuse notation and identify points, sets, functions and tensors on $U_p$ with their copies in $\psi_p(U_p)$, thus for example writing $U_p$ in place of $\psi_p(U_p)$, writing  $(y^1,\ldots,y^n)$ for points in $U_p$ and $g$ for the metric in these local coordinates on the set $U_p$ in place of $(\psi_p)_*g$. Finally, we will write $K$ for the support of the function $F_{q,\delta}$ and note that $K\Subset U_p$.
	
	Recalling the Definition~\ref{Gevrey_metric} and in view of the above remark, we may write
	\begin{equation}
	\label{metric_gevrey_exp} \|\partial^{\beta}_yg_{k\ell}\|_{L^{\infty}(K)}\leq A^{|\beta|+1}\,(\beta!)^N,
	\end{equation}
	for all multi-indexes $\beta\in \{0,1,2,\ldots\}^n$ and all $k,\ell=1,2,\ldots,n$, where
	$$ \partial_y^{\beta}:= \p^{\beta_1}_{y^1}\ldots \p^{\beta_n}_{y^n}\quad \text{and}\quad |\beta|=\sum_{\ell=1}^n\beta_\ell.$$
	Let us also observe that since $g$ is a Riemannian metric and since Gevrey classes of functions are closed under multiplication (see for example \cite{Gr}), the analogue of \eqref{metric_gevrey_exp} also holds for the the determinant of the metric as well as the elements of the inverse of the metric. In particular, we can deduce that there is a constant $A_1>0$ depending on $K$ and $g$ such that given any $k,\ell=1,2,\ldots,n$ and any $\beta \in \{0,1,2,\ldots\}^n$, there holds
	\begin{equation}
	\label{metric_gevrey_exp_2} \|\partial^{\beta}_yL^{k\ell}\|_{L^{\infty}(K)}\leq A_1^{|\beta|+1}\,(\beta!)^N,
	\end{equation}
	with $L^{k\ell}=(\det g)^{\frac{1}{2}}g^{k\ell}$ and that
	\begin{equation}
	\label{metric_gevrey_exp_3} \|\partial^{\beta}_yT\|_{L^{\infty}(K)}\leq A_1^{|\beta|+1}\,(\beta!)^N,
	\end{equation}
	where $T=(\det g)^{-\frac{1}{2}}$. It is also straightforward to see from \eqref{F_def} and \eqref{chi_0} that there is a constant $A_2>0$ depending on $F_{q,\delta}$ such that
	\begin{equation}
	\label{F_prop_der} \|\partial^{\beta}_yF_{q,\delta}\|_{L^{\infty}(K)}\leq A_2^{|\beta|+1}\,(\beta!)^{N'},
	\end{equation}
	  for all multi-indexes $\beta\in \{0,1,2,\ldots\}^n$. 
	  
	  We return to the expression $\Delta_g^{(m)}F_{q,\delta}$ and observe from its definition that it is a sum of $n^{2m}$ terms each of which has the form
	  $$Q^{j_1\ldots,j_m}_{k_1\ldots k_m}=T\p_{y^{j_1}}(L^{j_1k_1}\p_{y^{k_1}}(T\p_{y^{j_2}}(L^{j_2k_2}\p_{y^{k_2}}(\ldots \p_{y^{j_m}}(T(L^{j_mk_m}\p_{y^{k_m}}F_{q,\delta}))\ldots)$$
	 for some fixed $j_1,\ldots,j_m=1,\ldots,n$ and $k_1,\ldots,k_m=1,\ldots,n$. It is straightforward to see that
	 \begin{equation}
	 \label{Q_def}
	 \|Q^{j_1\ldots,j_m}_{k_1\ldots k_m}\|_{L^{\infty}(K)}\leq \sum_{i_1+i_2+i_3=2m}\frac{(2m)!}{i_1!i_2!i_3!}H_{i_1}I_{i_2}J_{i_3},\end{equation}
	 where 
	 \begin{equation}
	 \label{I_estimate}
	 H_{i_1}= \sum_{\ell_1+\ldots+\ell_m=i_1}\frac{i_1!}{\ell_1!\ldots\ell_m!}\widetilde{L}^{j_1k_1}_{\ell_1}\ldots \widetilde{L}^{j_mk_m}_{\ell_m},
	 \end{equation}
	 with 
	 $$\widetilde L^{j_sk_s}_{\ell_s}=(\sup_{|\beta|=\ell_s}\|\p^{\beta}_yL^{j_sk_s}\|_{L^{\infty}(K)})\quad s=1,\ldots,m,$$
	 and 
	 \begin{equation}
	 \label{J_estimate} I_{i_2}=\sum_{r_1+\ldots+r_m=i_2}\frac{i_2!}{r_1!\ldots r_m!}\widetilde{T}_{r_1}\ldots \widetilde{T}_{r_m},
	 \end{equation}
	 with 
	 $$\widetilde{T}_{r_s}=
	 (\sup_{|\beta|=r_s}\|\p_y^{\beta}T\|_{L^{\infty}(K)})\quad s=1,\ldots,m,$$
and finally, 
\begin{equation}
\label{K_estimate}
J_{i_3}=\sup_{|\beta|=i_3}\|\p_y^\beta F_{q,\delta}\|_{L^{\infty}(K)}.
\end{equation}
	Applying estimates \eqref{metric_gevrey_exp_2}-\eqref{metric_gevrey_exp_3} and \eqref{F_prop_der} in the latter three equations (and recalling that $N<N'$), it follows that
	\begin{multline*}
	H_{i_1}=\sum_{\ell_1+\ldots+\ell_m=i_1}\frac{i_1!}{\ell_1!\ldots\ell_m!}\widetilde{L}^{j_1k_1}_{\ell_1}\ldots \widetilde{L}^{j_mk_m}_{\ell_m}\leq\\ \sum_{\ell_1+\ldots+\ell_m=i_1}\frac{i_1!}{\ell_1!\ldots\ell_m!}A_1^{m+i_1}(\ell_1!)^{N'}\ldots (\ell_m!)^{N'}\leq A_1^{m+i_1}{m+i_1-1\choose i_1}(i_1!)^{N'}\\
	\leq (2A_1)^{m+i_1} (i_1!)^{N'}.
		\end{multline*}
		Analogously, we have
		$$I_{i_2}\leq (2A_1)^{m+i_2}(i_2!)^{N'}$$
		and that
		$$J_{i_3}\leq A_2^{1+i_3}(i_3!)^{N'}.$$
	Combining the latter three estimates with \eqref{Q_def} yields that 
	\begin{multline*}
	\|Q^{j_1\ldots,j_m}_{k_1\ldots k_m}\|_{L^{\infty}(K)}\leq \sum_{i_1+i_2+i_3=2m}\frac{(2m)!}{(i_1!)^{1-N'}(i_2!)^{1-N'}(i_3!)^{1-N'}}A_1^{2m+i_1+i_2}A_2^{1+i_3}\\
	\leq {2m+2\choose 2} (\max\{A_1,A_2\})^{4m}((2m)!)^{N'}.
	\end{multline*}
	Hence,
	$$\|\Delta_g^{(m)} F_{q,\delta}\|_{L^{\infty}(K)}\leq n^{2m}{2m+2\choose 2}(\max\{A_1,A_2\})^{4m}((2m)!)^{N'},$$
	for all $m\in \N$, which concludes the proof. 	
\end{proof}

\begin{lemma}
	\label{zeta_growth}
	Let $q\in U_p$ and let $\delta_0(q)>0$ be as defined by \eqref{delta_0}. Let $\delta \in (0,\delta_0(q))$ and let $f=F_{q,\delta}\in \mathcal C^{\infty}_0(U_p)$ be as defined by \eqref{F_def}. There holds,
	$$ \zeta_f^{(1)}(z,x)=\zeta_f^{(2)}(z,x) \quad \forall \, (z,x) \in \C\times \O.$$ 
\end{lemma}

\begin{proof}
	We fix $f=F_{q,\delta}$ as in the statement of the lemma. Let us define 
	\begin{equation}
	\label{zeta_difference_def} \zeta_f(z,x)=\zeta_f^{(1)}(z,x)-\zeta_f^{(2)}(z,x)\quad \forall\, (z,x)\in \C\times \O.\end{equation}
	Let us define for each $j=1,2$ and each $\mu\in \N$, the partial sums $S_\mu^{(j)}(z,x)$ and $\widetilde{S}_\mu^{(j)}(z,x)$ on $\C\times \O$ to be given by \eqref{S_mu}--\eqref{S_mu_abs} respectively and recall that
	\begin{equation}\label{S_abs_bound}|S_\mu^{(j)}(z,x)| \leq \widetilde{S}_\mu^{(j)}(z,x),\quad \forall \,(z,x) \in \C\times \O.\end{equation}
	Let us also define 
	$$ S_\mu(z,x)=S_\mu^{(1)}(z,x)- S_\mu^{(2)}(z,x),$$
	and record that
	$$ \zeta_f(z,x)=\lim_{\mu\to \infty} S_\mu(z,x).$$
    	Throughout the remainder of this proof, we use the notation $C>0$ for a generic positive constant that  may change from line to line. We claim that
	\begin{equation}
	\label{ineq_claim_1}
	\|\zeta_f(is,\cdot)\|_{L^{\infty}(\O)}\leq C,\quad \forall\, s\in \R,
	\end{equation}
	for some constant $C>0$ independent of $s$. To prove the claim \eqref{ineq_claim_1} we first add the two bounds \eqref{S_bound} for $j=1,2$ with $z=is$. Recalling the bound \eqref{S_abs_bound}, together with the definition $\kappa(is)=n+1$ and that $f=F_{q,\delta}\in \mathcal C^{\infty}_0(U_p)$, it follows that
	$$\|S_{\mu}(is,\cdot)\|_{L^{\infty}(\O)}\leq C\,\|f\|_{W^{2(n+1),\infty}(U_p)}\quad \forall\, s\in \R,$$
	for some constant $C>0$ that is independent of $s$ and $\mu$. By taking the limit as $\mu$ approaches infinity, we conclude the proof of \eqref{ineq_claim_1}. Next, we proceed to prove the bound 
		\begin{equation}
	\label{ineq_claim_2}
    \|\zeta_f(z,\cdot)\|_{L^{\infty}(\O)}\leq C e^{2N'' \,\textrm{Re}(z)\log(\textrm{Re}(z))},\quad \forall\,z \in \{w \in \C\,:\, \textrm{Re}(w)>0\},
	\end{equation}
	for any $N''\in (N',2)$ where $C>0$ is a constant that is independent of $z$. Recall also that $N'$ is as defined via \eqref{N'}. In order to prove this claim, we write $s:=\textrm{Re}(z)\geq 0$ and apply the bound \eqref{S_bound} again together with \eqref{S_abs_bound} to deduce that 
	$$\|S_{\mu}(z,\cdot)\|_{L^{\infty}(\O)}\leq C\,\|\Delta_g^{(\kappa(z))}f\|_{L^{\infty}(U_p)},\quad \forall\,s=\textrm{Re}(z)\geq 0,$$
	for some constant $C>0$ independent of $z$ and $\mu$, where we recall again that $f=F_{q,\delta}\in \mathcal C^{\infty}_0(U_p)$. Noting that
	\begin{equation}
	\label{kappa_pos}
	\kappa(z)=\lceil s\rceil +(n+1),\end{equation}
	for $s\geq 0$ and applying Lemma~\ref{technical_lemma} with $m=\kappa(z)$, we deduce that
	$$\|S_{\mu}(z,\cdot)\|_{L^{\infty}(\O)}\leq C^{\kappa(z)}\,((2\kappa(z))!)^{N'},$$
	for some $C>0$ independent of $z$ and $\mu$. The latter bound reduces as follows (keeping in mind that $C$ stands for a generic positive constant independent of $\mu$ and $z$),
	\begin{multline*}
	\|S_{\mu}(z,\cdot)\|_{L^{\infty}(\O)}\leq C^{\kappa(z)}\,((2\kappa(z))!)^{N'}\leq C^{\kappa(z)}e^{2N'\kappa(z)\log(2\kappa(z))}\\
	\leq C^{\kappa(z)}e^{2N'\kappa(z)\log \kappa(z)}\leq C^{s+1}e^{2N's\log s}\leq C\,e^{2N'' s\log s},
	\end{multline*}
	for all $s>0$, where we recall that $N''$ is any number in the interval $(N',2)$, and $C>0$ is independent of $z$. This concludes the proof of the bound \eqref{ineq_claim_2}. We also state the following bound which follows immediately from \eqref{ineq_claim_2},	
		\begin{equation}
	\label{ineq_claim_3}
	\|\zeta_f(z,\cdot)\|_{L^{\infty}(\O)}\leq Ce^{|z|^{2-c}}, \quad \forall z \in \{w \in \C\,:\, \textrm{Re}(w)\geq 0\},
	\end{equation}
for some constants $C>0$ and $c>0$ independent of $z$.

We are now ready to apply a variant of Carlson's theorem to finish the proof; Given any $x\in \O$ and in view of \eqref{ineq_claim_1}--\eqref{ineq_claim_2} and \eqref{ineq_claim_3} together with Lemma~\ref{lem_eq_seq} it follows that the hypothesis of Proposition~\ref{prop_carlson} is satisfied with $h(z):=\zeta_f(z,x)$ and $\tau=N''$. Therefore, $\zeta_f(z,x)=0$ on the right half plane $\{z\in \C\,:\, \textrm{Re}(z)\geq 0\}$ and subsequently by analytic continuation it must vanish on the entire complex plane. Thus, $\zeta_f(z,x)=0$ for all $(z,x)\in \C\times \O.$
\end{proof}

The latter lemma can be used to prove Proposition~\ref{main_prop}, and the key step in achieving this will be presented in Lemma~\ref{zeta_smooth_equal}, but first we will need the following auxiliary lemma.

\begin{lemma}
	\label{lem_pi_1}
	Let $k\in \N$ be fixed. There is some $q_0 \in U_p$, some $0<\delta<\delta_0(q_0)$ (with $\delta_0(q_0)$ as in \eqref{delta_0}) and some $x_0\in \O$ such that
	$$ (\pi_k^{(1)}F_{q_0,\delta})(x_0) \neq 0.$$
	A similar statement also holds for $\pi_k^{(2)}$.
\end{lemma}

\begin{proof}
	We give a proof by contradiction and assume that given any $q\in U_p$, any $\delta \in (0,\delta_0(q))$ and any $x\in \O$, there holds
	$$ (\pi_k^{(1)}F_{q,\delta})(x)=0.$$
	But then 
	\begin{equation}
	\label{zero_proj}
	(\pi_k^{(1)}f)|_{\O}=0 \qquad \forall\, f\in \mathcal C^{\infty}_0(U_p),
	\end{equation}
	thanks to \eqref{proj_F_f}. The latter equality implies that the restriction of the eigenspace of $\lambda_k^{(1)}$ to the set $U_p$ is identically zero and thus, in particular, $\phi^{(1)}_{k,1}\equiv 0$ on the set $U_p$. As, 
	$$ -\Delta_{g_1}\phi_{k,1}^{(1)}= \lambda_k^{(1)}\phi_{k,1}^{(1)}\qquad \text{on $M_1$},
	$$
	we conclude from the unique continuation principle for elliptic equations together with connectedness of $M_1$ that $\phi_{k,1}^{(1)}\equiv 0$ everywhere on $M_1$, which is a contradiction to $\|\phi^{(1)}_{k,1}\|_{L^2(M_1)}=1$.  
\end{proof}

\begin{lemma}
	\label{zeta_smooth_equal}
	Given any $k \in \N$, there holds, 
	$$\lambda_k^{(1)}=\lambda_k^{(2)},$$
	and 
	$$(\pi_k^{(1)}f)|_{\O}=(\pi_k^{(2)}f)|_{\O}\qquad \forall\, f\in \mathcal C^{\infty}_0(U_p).$$
\end{lemma}

\begin{proof}
	For $j=1,2$, let us also define 
	$$\mathbb A_{j}= \{\lambda_1^{(j)},\lambda_2^{(j)},\lambda_3^{(j)},\ldots\},$$
	and write $\mathbb A= \mathbb A_1\cup \mathbb A_2$. 
	We define, for each $j=1,2$, each $f\in C^{\infty}_0(\O^{\textrm{int}})$  and any $z\in \C\setminus \mathbb A$, the function
	\begin{equation}
	\label{R_lambda_1} \mathcal R_f^{(j)}(z,x)=\sum_{k=1}^{\infty}\frac{1}{\lambda^{(j)}_k-z}(\pi^{(j)}_kf)(x),\qquad \forall\,x\in \O.
	\end{equation}
	To rigorously justify the above the definition we recall that in view of the estimate \eqref{proj_est_1} with $m=n+1$ together with the bound \eqref{eigen_inf_sum} we have
	\begin{equation}
	\label{sum_pi_bound}
	\sum_{k=1}^{\infty}\|\pi^{(j)}_kf\|_{L^{\infty}(\O)}\leq C_j\|f\|_{W^{2(n+1),\infty}(\O)},
	\end{equation}
	for some constant $C_j>0$. Thus \eqref{R_lambda_1} is well defined and in fact given any $x\in \O$ and $f\in \mathcal C^{\infty}_0(\O^{\textrm{int}})$, the function $\mathcal R_f^{(j)}(\cdot,x)$ is holomorphic in $\C\setminus \mathbb A$ and has (possibly) simple poles on the set $\mathbb A$. The reason that we state the word possibly here is that if $(\pi_k^{(j)}f)(x_0)=0$ for some $f$ and some $x_0 \in \O$, then $\mathcal R_{f}^{(j)}(z,x_0)$ will be holomorphic at $z=\lambda_k$. Observe that given any $j=1,2$, any $f\in \mathcal C^{\infty}_0(\O^{\textrm{int}})$ and any $z \in \C$ with
	\begin{equation}
	\label{lambda_range} |z|<\min\{\lambda_1^{(1)},\lambda^{(2)}_1\},
	\end{equation}
	there holds,
\begin{equation}
\label{R_1_equal}
	\begin{aligned}
\mathcal R_f^{(j)}(z,x)&=\sum_{k=1}^{\infty}\frac{1}{\lambda_k^{(j)}-z}(\pi^{(j)}_kf)(x)=\sum_{k=1}^{\infty}\frac{(\pi^{(j)}_kf)(x)}{\lambda_k^{(j)}(1-\frac{z}{\lambda_k^{(j)}})}\\
&=\sum_{k=1}^{\infty}\sum_{m=1}^{\infty}\frac{z^{m-1}}{(\lambda_k^{(j)})^m}(\pi^{(j)}_kf)(x)=\sum_{m=1}^{\infty}z^{m-1}\sum_{k=1}^{\infty}\frac{(\pi^{(j)}_kf)(x)}{(\lambda_k^{(j)})^m}\\
&=\sum_{m=1}^{\infty}z^{m-1}\,\zeta_f^{(j)}(-m,x),\qquad \forall\, x \in \O.
\end{aligned}
\end{equation}
	We remark that in view of Fubini's theorem the interchanging of the sums in the second step above are justified as the series given by the double infinite sum above is absolutely convergent, thanks to \eqref{sum_pi_bound} and \eqref{lambda_range}. 
	
	Next, let us define the space of functions
	$$\mathcal X:=\{h\in \mathcal C^{\infty}_0(U_p)\,:\,\text{$h=F_{q,\delta}$ for some $q\in U_p$ and some $\delta \in (0,\delta_0(q))$}\}.$$
	We deduce, in view of the equality \eqref{R_1_equal} together with applying Lemma~\ref{zeta_growth} that 
	\begin{equation}
	\label{R_12_prelim}
	\mathcal R_f^{(1)}(z,x)=\mathcal R_f^{(2)}(z,x),
	\end{equation}
	for all $f\in \mathcal X$, $x\in \O$ and $z\in \C$ satisfying \eqref{lambda_range}. Finally, as the functions $\mathcal R_f^{(j)}(z,x)$ are holomorphic away from the set $\mathbb A$ which is a countable set of isolated points, we conclude via the analytic continuation principle that given any $f\in \mathcal X$ there holds
	\begin{equation}
	\label{R_12_equal}
		\mathcal R_f^{(1)}(z,x)=\mathcal R_f^{(2)}(z,x),\qquad \forall\, (z,x) \in (\C\setminus \mathbb A)\times \O. 
	\end{equation}
	
	In order to prove $\lambda_k^{(1)}=\lambda_k^{(2)}$ for all $k\in \N$, we will first prove the inclusion $\mathbb A_1\subset \mathbb A_2$. To this end, let us fix $k\in \N$, and note in view of Lemma~\ref{lem_pi_1} that there exists some $q_0 \in U_p$, some $0<\delta_1<\delta_0(q_0)$ with $\delta_0(q_0)$ as in \eqref{delta_0} and some $x_0\in \O$, such that 
	 \begin{equation}
	 \label{pi_1_1}
	(\pi_k^{(1)}F_{q_0,\delta_1})(x_0)\neq 0.\end{equation}
	Next, we apply \eqref{R_12_equal} with the choice $f=F_{q_0,\delta_1}\in \mathcal X$ and $x=x_0$. Recalling \eqref{pi_1_1}, we note that the function $\mathcal R_f^{(1)}(z,x_0)$ has a simple pole at $z=\lambda_k^{(1)}$. Consequently, in view of the equality \eqref{R_12_equal} together with the definition \eqref{R_lambda_1}, it follows that $\mathcal R_f^{(2)}(z,x_0)$ must also have a simple pole at $z=\lambda_k^{(1)}$, and therefore $\lambda_k^{(1)}\in \mathbb A_2$. As $k\in \N$ was arbitrary, we conclude that
	$\mathbb A_1\subset \mathbb A_2$. The opposite inclusion, namely $\mathbb A_2\subset \mathbb A_1$, follows by symmetry. 
	
	We have shown that $\lambda_k^{(1)}=\lambda_k^{(2)}$ for all $k\in \N$. In order to prove the second claim in the statement of the lemma, we note that given any $f\in \mathcal X$, we can use \eqref{R_12_equal} to write
    $$ (\pi_k^{(1)}f)(x)= \lim_{z\to \lambda_k^{(1)}}(\lambda_k^{(1)}-z)\mathcal R_f^{(1)}(z,x)= \lim_{z\to \lambda_k^{(2)}}(\lambda_k^{(2)}-z)\mathcal R_f^{(2)}(z,x)=(\pi_k^{(2)}f)(x),$$
    for all $x\in \O$ and all $k\in \N$. The second claim in the lemma now follows trivially from combining the latter identity with \eqref{proj_F_f}.     
		\end{proof}

\begin{lemma}
\label{proj_0_equal}
There holds:
$$d_k^{(1)}=d_k^{(2)}\quad \forall\, k\in \N,$$
and
$$ (\pi_0^{(1)}f)|_{\O}=(\pi_0^{(2)}f)|_{\O}\quad \forall\,f\in \mathcal C^{\infty}_0(U_p).$$
\end{lemma}

\begin{proof}
	Recall that
	$$
	d_k^{(j)} = \dim\{\pi_k^{(j)}f\,:\,f\in \mathcal C^{\infty}_0(U_p)\}\qquad \forall\, k\in \N.$$
	In order to prove the first claim, it suffices (thanks to Lemma~\ref{zeta_smooth_equal}) to show that
	\begin{equation}
	\label{d_equality}
	d_k^{(j)} = \dim\{(\pi_k^{(j)}f)|_{U_p}\,:\,f\in \mathcal C^{\infty}_0(U_p)\}\qquad \forall\, k\in \N.\end{equation}
	Indeed, if the latter identity was not true for some $j=1,2$ and some $k\in \N$, then there would exist a non-trivial function $\phi\in \mathcal C^{\infty}(M_j)$ with $\phi|_{U_p} \equiv 0$ that satisfies 
	$$-\Delta_{g_j}\phi = \lambda_k^{(j)}\phi,\qquad \text{on $M_j$}.$$
	But this is impossible, in view of the unique continuation principle for elliptic equations and the fact that $M_j$ is connected.
	
	To prove the second claim, we recall that $\lambda_0^{(1)}=\lambda_0^{(2)}=0$, and that the eigenspace corresponding to the zero eigenvalue is the set of constant functions on $M_j$. Therefore, $d_0^{(1)}=d_0^{(2)}=1$ and for each $j=1,2$, 
	$$\phi_{0,1}^{(j)}= \frac{1}{\sqrt{\textrm{Vol}(M_j,g_j)}},\qquad \text{for $j=1,2$,}$$
	where the notation $\textrm{Vol}(M_j,g_j)$ stands for the volume of the manifold $M_j$ with respect to the metric $g_j$. Thus, in order to prove the second claim in the lemma, it suffices to show that the two manifolds have the same volume.
	
	To this end, let us observe from the equalities  
	$\lambda_k^{(1)}=\lambda_k^{(2)}$ and  $d_k^{(1)}=d_k^{(2)}$ for all $k=0,1,\ldots$, that there holds
	$$\mathcal N_1(\lambda)=\mathcal N_2(\lambda)\qquad \forall \lambda \in \R,$$
	where we recall that the notation $\mathcal N_j(\lambda)$, $j=1,2,$ stands for the number of eigenvalues of $\Delta_{g_j}$ on $(M_j,g_j)$ that are less than $\lambda$ (counting with multiplicity). The equality of the volumes now follows the latter equality together with Weyl's asymptotic formula for $\mathcal N_j(\lambda)$, see for example \cite{Ivrii}.
\end{proof}

We are now ready to prove the main proposition. 
\begin{proof}[Proof of Proposition~\ref{main_prop}]
Equality of the eigenvalues and their multiplicities follow immediately from Lemma~\ref{zeta_smooth_equal} and Lemma~\ref{proj_0_equal}. The equality of the projection operators $\pi_kf$ with $f\in \mathcal C^{\infty}_0(\O^{\textrm{int}})$ follows from the same lemmas by using a partition of unity. Indeed, let us assume that $f\in \mathcal C^{\infty}_0(\O^{\textrm{int}})$, and write $K:=\supp f$. Given any point $p\in K$, and in view of the hypothesis of Proposition~\ref{main_prop}, let $(U_p,\psi_p)$ be a coordinate chart in which the components of the metric belong to the Gevrey class $\mathscr G^N$ with $N$ as in the hypothesis of Theorem~\ref{main_thm}. Then, since $K$ is compact there is a finite collection of points $I$ such that $\{U_p\}_{p\in I}$ gives an open cover for the set $K$. Let $\{\eta_p\}_{p\in I}$ denote a partition of unity subordinate to the open cover $\{U_p\}_{p\in I}$. Then, by definition of the projection operators \eqref{proj_j_def}, together with Lemma~\ref{zeta_smooth_equal} and Lemma~\ref{proj_0_equal}, we can write for any $k=0,1,2,\ldots,$ and any $x\in \O$,
 \begin{equation}
 \label{partition_unity_eq}
 \begin{aligned}
 (\pi_k^{(1)}f)(x)&=\sum_{\ell=1}^{d_k^{(1)}}(f,\phi_{k,\ell}^{(1)})_{L^2(\O)}\phi_{k,l}^{(j)}(x)=\sum_{\ell=1}^{d_k^{(1)}}\sum_{p\in I}(f\eta_p,\phi_{k,\ell}^{(1)})_{L^2(U_p)}\phi_{k,l}^{(j)}(x)\\
 &=\sum_{p\in I}(\pi_k^{(1)}\eta_pf)(x)=\sum_{p\in I}(\pi_k^{(2)}\eta_pf)(x)=  (\pi_k^{(2)}f)(x).
 \end{aligned}
\end{equation}
\end{proof}

\section{Proof of the main theorem via a standard reduction to the Gel'fand inverse spectral problem}

Proposition~\ref{main_prop} shows that the source-to-solution map for a closed connected Riemannian manifold $(M,g)$ subject to an observable region $\O$ uniquely determines the eigenvalues, their multiplicities and the restriction of the projection operators of each eigenspace to the observable set. Thus, to complete the proof of Theorem~\ref{main_thm} we must determine whether the latter spectral data determines the topology, differential structure and the Riemannian metric of the inaccessible region $M\setminus \O$. We remark that this problem is rather similar to the well known Gel'fand inverse spectral problem first posed by Gel'fand in \cite{Gelfand} which concerns finding the topology, differential structure and Riemannian metric of a compact manifold with boundary from the spectral data associated to the Dirichlet Laplacian on the boundary. This problem was first solved by Belishev and Kurylev in \cite{BK92}, see also the precursors \cite{Bel1,Bel2}. We refer the reader to \cite{KrKaLa} that solves an analogous problem stated on closed Riemannian manifolds, to \cite{AKKLT} for the solution to a variant of the problem and to the monograph \cite{KKL} for review. 

To complete our proof of Theorem~\ref{main_thm} we show that the spectral data obtained in Proposition~\ref{main_prop} uniquely determines the spectral data for the Neumann Laplacian for the inaccessible submanifold $M\setminus \O$ and then apply the main result of \cite{BK92} to conclude the proof. Let us remark that alternatively, one could prove Theorem~\ref{main_thm} by reducing our inverse problem to the one studied in \cite{KrKaLa} and then applying \cite[Theorem 1]{KrKaLa} to conclude the proof. 


Before proving Theorem~\ref{main_thm} we need to fix a few notations. We write $\Sigma:=\p\O$ to stand for the boundary of the observation submanifold $\O$ and note that since $M_j$, $j=1,2$ are connected, there holds
$$\Sigma=\p \O= \p (M_j\setminus\O).$$
For each $\lambda>0$, 
we consider the equation 
\begin{equation}\label{boundary_eq}
\left\{ \begin{array}{ll}-\Delta_{g_j} u_j+\lambda\,u_j=0, & \mbox{on}\ M_j\setminus \O,\\  
u_j  =h & \mbox{on}\ \Sigma:=\p\O,\\
\end{array} \right.
\end{equation}
Given any $h \in H^{\frac{3}{2}}(\Sigma)$, it is classical that the above problem admits a unique solution $u_j\in H^2(M_j\setminus\O)$. We define the Dirichlet-to-Neumann map associated to \eqref{boundary_eq} via the mapping
\begin{equation}\label{DN_map}
\Lambda^\lambda_{M_j,g_j}(h)=\p_\nu u_j|_{\Sigma} 
\end{equation}
where $\nu$ is the exterior unit normal vector field on $\Sigma$ (exterior with respect to $\O$) and the right hand side is to be understood as an element in $H^{\frac{1}{2}}(\Sigma)$.

Next, and for each $\lambda>0$, we consider the equation
\begin{equation}\label{source_eq}
 -\Delta_{g_j}v_j+\lambda\,v_j=f,\quad \mbox{on}\ M_j,  
\end{equation}
Given any $f \in L^2(\O)$ (this means that $f\in L^2(M_j)$ and the essential support of $f$ lies in $\O^{\textrm{int}}$), the above problem admits a unique solution $v_j\in H^2(M_j)$. We define the Source-to-Dirichlet map and the Source-to-Neumann map associated to \eqref{source_eq} via the mapping
\begin{equation}\label{SD_map}
\mathscr D^\lambda_{M_j,g_j}(f)=v_j|_{\Sigma}, 
\end{equation}
and
\begin{equation}\label{SN_map}
\mathscr N^\lambda_{M_j,g_j}(f)=\p_\nu v_j|_{\Sigma}
\end{equation}
respectively, where $v_j$ is the unique soution to \eqref{source_eq} with source $f$. It is straightforward to see that given $j=1,2$ and any $\lambda>0$, there holds
\begin{equation}
\label{ND_relation}  
\mathscr N^{\lambda}_{M_j,g_j}(f) =  \Lambda_{M_j,g_j}^\lambda\left( \mathscr D_{M_j,g_j}^\lambda(f)\right)\qquad \forall\, f \in L^2(\O).
\end{equation}
We have the following lemma.
\begin{lemma}
\label{lem_density}
	For $j=1,2$, any $\lambda>0$, and any $h \in \mathcal C^{\infty}(\Sigma)$, there exists $f_j\in L^2(\O)$ such that
	$$\mathscr D^{\lambda}_{M_j,g_j}f_j=h.$$
\end{lemma}

\begin{proof}
For $j=1,2$, let $w_j\in \mathcal C^{\infty}(\overline{M_j\setminus \O})$ be the unique solution to equation \eqref{boundary_eq} subject to the boundary data $h$. Let $\widetilde{w}_j$ be a smooth extension of $w_j$ to $M_j$ and define
$$f_j=(-\Delta_{g_j}+\lambda)\widetilde{w}_j.$$
It is clear that $f_j\in L^2(\O)$ as it is smooth everywhere and identically zero on $M_j\setminus\O$. Hence, by definition, we have $\mathscr D^{\lambda}_{M_j,g_j}f_j=h.$
\end{proof}

\begin{proof}[Proof of Theorem~\ref{main_thm}]
We recall the notation \eqref{O_g_equal} and define, for each $j=1,2$, each $f\in \mathcal C^{\infty}_0(\O^{\textrm{int}})$ and each $\lambda>0$, the function
\begin{equation}
\label{R_lambda_2}
\widetilde{\mathcal R}_f^{(j)}(\lambda,x)=\sum_{k=0}^{\infty}\frac{1}{\lambda^{(j)}_k+\lambda}(\pi^{(j)}_kf)(x),\qquad \forall\,x\in \O.
\end{equation}
Observe that the above definition is justified thanks to \eqref{sum_pi_bound}. 
Observe also that in view of Proposition~\ref{main_prop} together with the definition \eqref{R_lambda_2}, we have 
\begin{equation}
\label{R_equal_smooth}
 \widetilde{\mathcal R}_f^{(1)}(\lambda,x)= \widetilde{\mathcal R}_f^{(2)}(\lambda,x),\qquad \forall\, \lambda>0, 
 \end{equation}
for all $f\in \mathcal C^{\infty}_0(\O^{\textrm{int}})$ and all $x \in \O$.  It is straightforward to see from the definition \eqref{R_lambda_1} that given any $\lambda>0$ and any $f\in \mathcal C^{\infty}_0(\O^{\textrm{int}})$, there holds
 $$(\mathscr D^{\lambda}_{M_j,g_j}f)(x)= \widetilde{\mathcal R}_f^{(j)}(\lambda,x), \qquad \forall\, x\in \Sigma,$$
 and that
 $$ (\mathscr N^{\lambda}_{M_j,g_j}f)(x)= \p_\nu\widetilde{\mathcal R}_f^{(j)}(\lambda,x),\qquad \forall\, x\in \Sigma.$$
 Combining the latter two identities with \eqref{R_equal_smooth}, we deduce that
 \begin{equation}
 \label{D_map_equal}
 \mathscr D_{M_1,g_1}^{\lambda}f=  \mathscr D_{M_2,g_2}^{\lambda}f \qquad \text{on $\Sigma$},
 \end{equation}
 and
  \begin{equation}
 \label{N_map_equal}
 \mathscr N_{M_1,g_1}^{\lambda}f=  \mathscr N_{M_2,g_2}^{\lambda}f \qquad \text{on $\Sigma$},
 \end{equation}
for all $f\in \mathcal C^{\infty}_0(\O^{\textrm{int}})$ and all $\lambda>0$. By continuity, we conclude that the equations \eqref{D_map_equal} and \eqref{N_map_equal} must also hold for all $f\in L^2(\O)$. Finally, recalling the equation \eqref{ND_relation} together with Lemma~\ref{lem_density}, we conclude that
$$ \Lambda_{M_1,g_1}^{\lambda}h=  \Lambda_{M_2,g_2}^{\lambda}h,$$
for any $h\in C^{\infty}(\Sigma)$, and thus by continuity for all $h\in H^{\frac{3}{2}}(\Sigma)$. 

We have shown that the Dirichlet-to-Neumann maps for equation \eqref{boundary_eq} with $j=1,2$ are identical for all $\lambda>0$.  Applying \cite[Theorem 1]{KKLM}, it follows that the Dirichlet spectral data for the two manifolds $(M_1\setminus \O_1,g_1|_{M_1\setminus\O_1})$ and $(M_2\setminus \O_2,g_2|_{M_2\setminus\O_2})$ are identical and as such we have reduced our inverse problem to the standard Gel'fand spectral inverse problem studied in \cite{BK92}. Applying the main result in \cite{BK92}, we conclude that there is a diffeomorphism 
$$\Phi_1: M_1\setminus \O_1 \to M_2\setminus \O_2,$$
that is equal to identity on $\Sigma$ and such that
$$(\Phi_1^*g_2)|_{M_1\setminus\O_1}=g_1|_{M_1\setminus\O_1}.$$
Recalling \eqref{O_eq}, we can glue the diffeomorphism $\Phi_1$ along $\Sigma$ with the identity diffeomorphism in \eqref{O_eq} (see for example \cite[Theorem 2.8]{M}) to construct the desired diffeomorphism $\Phi$.
\end{proof} 

\appendix

\section{A variant of Carlson's theorem with faster than exponential growth along the real axis}
\label{sec_appendix}

Carlson's theorem in complex analysis gives sharp conditions for a holomorphic function in the right half plane of exponential type, so that the function may be uniquely determined by its values on the set of positive integers, see \cite{Carlson} for the original version of this theorem and to \cite{Fu,Rubel} for some extensions that allow for example more general sequences of numbers along the real axis. This theorem can be seen as an application of the Phragm\'{e}n--Lindel\"{o}f theorem in complex analysis.

For our purposes, we need a variant of Carlson's theorem that is due to Pila, see \cite{Pila}. Pila used the precise asymptotics of the Gamma function to derive a variant of Carlson's theorem that allows slightly faster growth for the holomorphic function along the real-axis at the expense of some slower growth on the imaginary axis. This slightly faster growth along the real axis that is captured by the logarithm function is what essentially allowed us to prove the main theorem. Pila's variant is stated for holomorphic functions that are known on the set of positive integers but as we need to consider more general sequences of real numbers (see \eqref{b_sequence}), we present here a slightly altered statement that is tailored to our setup and include the short proof for the sake of completeness. 
\begin{proposition}
	\label{prop_carlson}
	Let $\alpha\in (0,1)$ and let $(b_k)_{k=1}^{\infty}$ be the strictly monotone sequence defined by \eqref{b_sequence}. Write $z=x+iy$ and suppose that $h(z)$ is holomorphic in $x\geq 0$ and that there is constants $c_1,c_2>0$ such that
	\begin{enumerate}
	\item[(H1)]{$|h(iy)| \leq c_1$ for all $y\in \R$,} 
	\item[(H2)] {$\exists\,\tau \in (0,2)$ such that $|h(x)| \leq c_1 e^{2\tau x\log x}$ for all $x>0$,} 
	\item[(H3)]{$|h(z)|\leq c_1e^{|z|^{2-c_2}}$ for all $z\in \{x+iy\in \C\,:\,x\geq 0\}$.}
	\end{enumerate}
	 Suppose that $h(b_k)=0$ for all $k\in \N$. Then $h$ vanishes identically on the set $\{x+iy\in \C\,:\,x\geq 0\}$.  
\end{proposition}

\begin{proof}
	We follow the argument of \cite{Pila} with small modifications. Let 
	$$\tau' \in (\tau,2).$$
	We write $\Gamma(z)$ for the Gamma function and note 
	$\Gamma(z+1)$ is regular and never zero in $x\geq 0$. Thus, $\log\Gamma(z+1)$ is well defined and as such given any $\beta \in \C$, the function $\Gamma(z+1)^\beta$ can also be defined in a regular manner there. Note that we are taking the principal value of the logarithm function. In view of Sterling's formula we have on $x>0$,
		\begin{equation}
	\label{sterling}
	\log\Gamma(z+1)=(z+\frac{1}{2})\log z-z+\frac{1}{2}\log(2\pi)+O(|z|^{-\frac{1}{2}}),	
	\end{equation}
	as $z\to \infty$, uniformly in $x\geq 0$. We define the function
	$$f(z)= \frac{h(z)}{\Gamma(z+1)^{2\tau'}} \quad \text{on $x\geq 0$.}$$
	The function $f(z)$ is regular in $x\geq 0$ and in view of condition (H2) together with Sterling's approximation \eqref{sterling}, it is also uniformly bounded on the positive $x$-axis. On the other hand, by (H1) and \eqref{sterling}, there holds
	\begin{equation}
	\label{y_axis_bound}
	\limsup_{|y|\to \infty} \frac{\log|f(iy)|}{|y|}\leq \pi\,\tau'<2\pi.\end{equation}
	Furthermore, by (H3) and \eqref{sterling}, there holds for any $0<c_3<c_2$,
	\begin{equation}
	\label{awayfrom_axis}
	|f(z)|=O(e^{|z|^{2-c_3}}),
	\end{equation}
	as $|z|\to \infty$, uniformly in the argument of $z$. Now let $\tau'' \in (\tau',2)$ and subsequently consider the functions
	$$k_{+}(z)=f(z) e^{i\pi\tau'' z}\qquad \text{on the first quadrant},$$
	and
		$$k_{-}(z)=f(z) e^{-i\pi\tau''\,z}\qquad \text{on the fourth quadrant},$$
	in the complex plane. The function $k_+$ (resp. $k_-$) is uniformly bounded on the positive $x$-axis and also on the positive (resp. negative) $y$-axis, and there holds
	$$|k_{\pm}(z)|= O(e^{|z|^{2-c_4}}),$$
for some $0<c_4<c_3$, as $|z|\to \infty$, uniformly in the argument of $z$. By Phragm\'{e}n--Lindel\"{o}f (see for example \cite[Chapter 5, Section 6]{Ti}) the function $k_{+}(z)$ (resp. $k_-(z)$) will be uniformly bounded in the first quadrant (resp. fourth quadrant). We deduce that the function $f(z)$ will be of exponential type throughout $x\geq 0$, that is to say $|f(z)|\leq C\,e^{\pi\tau''|z|}$ throughout $x\geq 0$. Note that since $\alpha \in (0,1)$, there holds
\begin{equation}
\label{b_gap}
b_{k+1}-b_{k}\geq \min\{\alpha,1-\alpha\}>0.
\end{equation}
Recalling that $f$ satisfies \eqref{y_axis_bound} on the $y$-axis and additionally it vanishes on the sequence of points $\{b_k\}_{k\in \N}$ on the $x$-axis, we are in the setup of a classical variant of Carlson's theorem with non-integer points of vanishing, see for example  \cite[Theorem 1]{Fu} (with $\psi(r)=r^4$, $k=\pi\tau''$ and $c=\frac{1}{2}\min\{\alpha,1-\alpha\}>0$ in that theorem) and \cite[Theorem 5]{Rubel} (with $\mathfrak G=2\pi$, $\gamma=\min\{1-\alpha,\alpha\}$ and $L(\Lambda)=2$ in that theorem). We conclude that $f$ and consequently $h$ vanish identically on the set $\{x+iy\in \C\,:\,x\geq 0\}$ (we remark that in \cite{Rubel} the results are stated for entire holomorphic functions but all the arguments work when restricting to the right half plane). 
\end{proof}

\end{document}